\documentclass[a4paper,12pt]{article}
\usepackage[cp1250]{inputenc}
\usepackage[T1]{fontenc}
\usepackage{amsfonts}
\usepackage{enumerate}
\usepackage{color}
\usepackage{amsmath}
\usepackage{amssymb}
\usepackage{amsthm} 
\usepackage{textcomp}
\usepackage{wrapfig}
\usepackage{yfonts} 
\usepackage{url}
\usepackage{fancyhdr}
\usepackage{geometry}
\usepackage{stmaryrd}
\usepackage{bbding}
\usepackage{enumitem} 
\usepackage{bbm}
\everymath{\displaystyle}
\theoremstyle{plain} 
\newtheorem{tw}{Theorem}[section]	
\newtheorem{pr}{Problem}

\newtheorem{corollary}{Corollary}[section]

\theoremstyle{definition} 

\newtheorem{remark}{Remark}[section]

\newcommand\cA{{\mathcal A}}
\newcommand\cF{{\mathcal F}}
\newcommand\cS{{\mathcal S}}
\newcommand\cM{{\mathcal M}}
\newcommand\fS{{\textfrak{S}}}

\newcommand\bIs{\mathbf{I}_{\rS}}
\newcommand\bI{\mathbf{I}}

\newcommand{\rS}{\mathrm{S}}
\newcommand{\rM}{\mathrm{M}}
\newcommand{\rC}{\mathrm{C}}

\renewcommand\ge{\geqslant}
\renewcommand\le{\leqslant}

\newcommand{\mI}[1]{\mathbbm{1}_{#1}}

\begin{document}
\begin{center} 
{\LARGE On some properties of seminormed fuzzy integrals}
\end{center}

\begin{center}
Michał Boczek
\footnote{Corresponding author. E-mail adress: 800401@edu.p.lodz.pl; 
}, 
Marek Kaluszka
\footnote{E-mail adress: kaluszka@p.lodz.pl; tel.: +48 42 6313859; fax.: +48 42 6363114.}
\\
{\emph{\small{Institute of Mathematics, Lodz University of Technology, 90-924 Lodz, Poland}}}
\end{center}


\begin{abstract} 
We give solutions to  Problems $2.21$, $2.31$ and $2.32$,  which were posed 
Borzowá-Molnárová,  Hal\v{c}inová and Hutník in [{\it The smallest semicopula-based universal integrals I: properties and characterizations,} Fuzzy Sets and Systems (2014),  http://dx.doi.org/
10.1016/j.fss. 2014.09.0232014].
\end{abstract}


\section{Introduction}
Let $(X,\cA)$ be a~measurable space, where $\cA$ is a~$\sigma$-algebra of subsets of a non-empty set $X,$ and let $\cS$ be the family of all measurable spaces. The~class of all $\cA$-measurable functions $f\colon X\to [0,1]$ is denoted by $\cF_{(X,\cA)}.$ A~{\it capacity} on $\cA$  is a~non-decreasing set function 
$\mu\colon \cA\to [0,1]$ with $\mu(\emptyset)=0$ and  $\mu(X)=1.$
We denote by $\cM_{(X,\cA)}$ the class of all 
capacities  on $\cA.$  

Suppose that $\rS\colon [0,1]^2\to [0,1]$ is a~semicopula (also called 
a~$t$-{\it seminorm}), i.e., a~non-decreasing function in both coordinates with the neutral element equal to $1.$  It is clear that $\rS(x,y)\le x\wedge y$ 
and $\rS(x,0)=0=\rS(0,x)$ for all $x,y\in [0,1],$ where $x\wedge y=\min(x,y)$ (see $\cite{bas,dur,klement2}$).  We denote the class of all semicopulas
by $\fS$. 

There are three important examples of semicopulas: $\rM,$ $\Pi$ and $\rS_L,$ where $\rM(a,b)=a\wedge b,$ $\Pi(a,b)=ab$ and $\rS_L(a,b)=(a+b-1)\vee 0;$
$\rS_L$ is called the {\it Łukasiewicz t-norm}  $\cite{lukasiewicz}.$ Hereafter, 
$a\vee b=\max(a,b)$.
A~function  $\rS^*\colon [0,1]^2\to [0,1]$ 
is called a~{\it $t$-coseminorm} if  $\rS^*(x,y)=1-\rS(1-x,1-y)$ for some 
semicopula $\rS$.

A generalized Sugeno integral is defined by
\begin{align*}
\bIs(\mu,f):=\sup_{t\in [0,1]} \rS\Big(t,\mu\big(\lbrace f\ge t\rbrace \big)\Big),
 \end{align*} 
where $\lbrace f\ge t\rbrace=\lbrace x\in X\colon f(x)\ge t\rbrace,$  $(X,\cA)\in \cS$ and $(\mu, f)\in \cM_{(X,\cA)}\times \cF_{(X,\cA)}.$ In the literature the functional $\bIs$ is also called {\it seminormed fuzzy integral} $\cite{suarez,klement3,ouyang3}.$
Replacing semicopula $\rS$ with $\rM$, we get the {\it Sugeno integral} $\cite{sugeno1}$. Moreover, if  $\rS=\Pi,$ then $\bI_{\Pi}$ is called the {\it Shilkret integral} $\cite{shilkret}.$

The paper is devoted to the solution 
of Problems $2.21$, $2.31$ and $2.32$,  which were posed by Borzowá-Molnárová,  Hal\v{c}inová and Hutník $\cite{hutnik}$
(see also $\cite{mesiar11}$). In Section 2 we  present our main results and some related results. In the next section we give conditions equivalent to comonotone $\circ$-additivity of the generalized Sugeno integral $\bIs.$

 \section{Weak subadditivity of integral $\bIs$}
In this section we present necessary and sufficient conditions for validity of the inequality 
\begin{align}\label{m1}
  \bIs\big(\mu,f+a\big)\le \bIs \big(\mu,f\big)+a,
  \end{align}
 where $f+a\in [0,1].$ 
Borzowá-Molnárová {\it et al.} $\cite{hutnik}$ proposed a sufficient condition only, which turns out to be too restrictive.

\begin{tw}
The inequality $\eqref{m1}$ is satisfied for all $a\in [0,1]$  and $(\mu, f)\in\cM_{(X,\cA)}\times\cF_{(X,\cA)}$ such that $f+a\in [0,1]$ iff
\begin{align}\label{m0}
\rS(c+a,b)\le \rS(c,b)+a
\end{align}
 for all $a,b,c\in [0,1]$ such that $a+c\in [0,1].$ 
\end{tw} 

 \begin{proof}
 We first prove that the inequality $\eqref{m1}$ holds true under the assumption 
 (\ref{m0}).  The inequality can be rewritten  as follows
  \begin{align}\label{m1a}
 \sup_{t\in [0,1]} \rS\Big(t,\mu\big(\lbrace f+a\ge t\rbrace \big)\Big)\le \sup_{t\in [0,1]} \rS\Big(t,\mu\big(\lbrace f\ge t\rbrace \big)\Big)+a.
 \end{align}
The inequality $\eqref{m1a}$ is obvious for $a\in \{0,1\}$, so we assume that $a\in (0,1).$ Observe that 
$\mu\big(\lbrace f+a\ge t\rbrace \big)=1$ for $t\in [0,a].$ 
Since $\rS(a,1)=a$ for all $a$,  the left-hand side of $\eqref{m1a}$ takes the form
\begin{align}\label{e2}
\sup_{t\in [0,1]} \rS\Big(t,\mu\big(\lbrace f+a\ge t\rbrace \big)\Big)&=\max\Big[ \sup_{t\in [0,a]} \rS(t,1), \sup_{t\in (a,1]} \rS\Big(t,\mu\big(\lbrace f+a\ge t\rbrace \big)\Big)\Big]\nonumber\\&=a\vee \sup_{t\in (0,1-a]} \rS\Big(t+a,\mu\big(\lbrace f\ge t\rbrace \big)\Big)\nonumber\\&=\sup_{t\in [0,1-a]} \rS\Big(t+a,\mu\big(\lbrace f\ge t\rbrace \big)\Big).
\end{align}
 From $\eqref{m1a}$ and $\eqref{e2}$ we get
 \begin{align}\label{m2}
 \sup_{t\in [0,1-a]} \rS\Big(t+a,\mu\big(\lbrace f\ge t\rbrace \big)\Big)\le \sup_{t\in [0,1-a]} \rS\Big(t,\mu\big(\lbrace f\ge t\rbrace \big)\Big)+a,
 \end{align}
 because $\lbrace f\ge t\rbrace=\emptyset$ for $t>a$ and $\rS(b,0)=0$ for all $b.$ Let $f(x)=cx,$ $x\in [0,1],$ and $\mu(A)=b$ for all $A\notin \{\emptyset,X\},$ where $0< b< 1,$  $0<c\le 1-a.$ Thus 
\begin{align*}
 \sup_{t\in [0,1-a]} \rS\Big(t+a,\mu\big(\lbrace f\ge t\rbrace \big)\Big)&=a\vee \big(\rS(c+a,b)\big),\\
 \sup_{t\in [0,1-a]} \rS\Big(t,\mu\big(\lbrace f\ge t\rbrace \big)\Big)+a&=\rS(c,b)+a=a\vee \big(\rS(c,b)+a\big).
 \end{align*} 
Combining this with $\eqref{m2}$ we obtain
 \begin{align}\label{m3}
 a\vee \big(\rS(c+a,b)\big)\le a\vee \big(\rS(c,b)+a\big)
 \end{align}
 for all $(a,b,c)\in\rC,$ as $\rC=\lbrace (a,b,c)\colon 0<a<1,\, 0< b< 1,\, 0<c\le 1-a \rbrace.$ 
 
 It is obvious that $x\vee y \le x\vee z$ implies ($y\le z$) or $(x\ge y>z)$ for all $x,y,z$, so from (\ref{m3}) it follows that
\begin{align}\label{m2a}
\bigl(\rS(c+a,b)\le \rS(c,b)+a\bigr)\quad  \textrm{or}\quad \big( a\ge \rS(c+a,b) \quad \textrm{and}\quad \rS(c+a,b)>\rS(c,b)+a\big).
\end{align}
 Hence we get
\begin{align}\label{m3a}
 \rS(c+a,b)\le \rS(c,b)+a,
 \end{align} 
 for all $a,b,c\in \rC,$ as the second sentence in $\eqref{m2a}$ leads to a~contradiction. Let us notice that $\eqref{m3a}$ also holds for $a=0,$ $a=1$ and $c=0,$ as well as $b=0,$ $b=1.$ 
Now we show that if $\rS(c+a,b)\le \rS(c,b)+a$, then inequality $\eqref{m1}$ is satisfied for all $(\mu, f)\in\cM_{(X,\cA)}\times\cF_{(X,\cA)}$ such that $f+a\in[0,1].$ Applying $\eqref{m0}$ we obtain
  \begin{align*}
  \bIs(\mu,f+a)=\sup_{t\in [0,1-a]} \rS\Big(t+a,\mu\big(\lbrace f\ge t\rbrace \big)\Big)\le \sup_{t\in [0,1-a]} \rS\Big(t,\mu\big(\lbrace f\ge t\rbrace \big)\Big)+a=\bIs(\mu,f)+a.
  \end{align*}
  The proof is complete.
 \end{proof}

\begin{remark}
From $\eqref{m0}$ it follows that $\rS_L(x,y)\le \rS(x,y)$ for all $x,y\in [0,1];$ to see this  put $c+a=1$ in $\eqref{m0}$.
\end{remark}
 
We show that the sufficient condition guaranteeing validity of  $\eqref{m1}$ (see the inequality (3) in Borzowá-Molnárová 
{\it et al.} $\cite{hutnik}$) is also necessary but for 
a slightly different problem (see Corollary $\ref{c1}\,(a)$ below).
 
 \begin{tw}\label{tw3} The inequality  \begin{align}\label{s2}
  \bI_{\rS_1}\big(\mu,(f+a)\mI{A}\big)\le \bI_{\rS_2}\big(\mu,f\mI{A}\big)+\bI_{\rS_3}\big(\mu,a\mI{A}\big)
  \end{align}
holds for all $a\in [0,1]$, 
$\rS_i\in\fS,$  $A\in\cA,$ and $(\mu, f)\in \cM_{(X,\cA)}\times \cF_{(X,\cA)}$ such that  $f+a\in [0,1]$ iff the inequality
 \begin{align}\label{s1}
  \rS_1(x+y,z)\le \rS_2(x,z)+\rS_3(y,z)
  \end{align} 
  is satisfied for all $x,y,z\in [0,1]$ such that $x+y\in [0,1].$ 
 
 \end{tw}
 
\begin{proof}
We claim that  $\eqref{s2}$ follows from $\eqref{s1}$. Indeed 
\begin{align*}
\sup_{t\in [0,1]} \rS_1\Big(t,\mu\big(A\cap\lbrace f+a\ge t\rbrace &\big)\Big)=\sup_{t\in [0,a]} \rS_1\big(t,\mu(A)\big)\vee\sup_{t\in (a,1]} \rS_1\Big(t,\mu\big(A\cap\lbrace f+a\ge t\rbrace \big)\Big)\\&=\rS_1\big(a,\mu(A)\big)\vee \sup_{t\in(0,1-a]} \rS_1 \Big(t+a,\mu\big(\lbrace f\ge t\rbrace \big)\Big)\\&=\sup_{t\in [0,1-a]} \rS_1 \Big(t+a,\mu\big(\lbrace f\ge t\rbrace \big)\Big)\\&\le \sup_{t\in [0,1-a]} \bigg(\rS_2\Big(t,\mu\big(A\cap\lbrace f\ge t\rbrace \big)\Big)+ \rS_3\Big(a,\mu\big(A\cap\lbrace f\ge t\rbrace \big)\Big)\bigg)\\&\le\bI_{\rS_2}\big(\mu,f\mI{A}\big)+\bI_{\rS_3}\big(\mu,a\mI{A}\big).
\end{align*}
Putting $f=b\mI{A}\in [0,1]$ in $\eqref{s2}$, we obtain  $\eqref{s1},$ which completes the proof.
\end{proof}

Applying  Theorem $\ref{tw3}$ to $\rS_1=\rS_2$ and $\rS_3=\rS$, $\rS_3=\wedge$ or $\rS_3=\Pi$ we get
 \begin{corollary}\label{c1}
 \begin{itemize}
 \item[(a)] $\rS(x+y,z)\le \rS(x,z)+\rS(y,z)$ for all $x,y,z$ iff for all $a,f,A,\mu$ 
 \begin{align*}
 \bIs\big(\mu,(f+a)\mI{A}\big)\le \bIs\big(\mu,f\mI{A}\big)+\bIs\big(\mu,a\mI{A}\big),
 \end{align*}
 \item[(b)]  $\rS(x+y,z)\le \rS(x,z)+(y\wedge z)$ for all $x,y,z$ iff for all $a,f,A,\mu$
  \begin{align*}
 \quad \bIs\big(\mu,(f+a)\mI{A}\big)\le \bIs\big(\mu,f\mI{A}\big)+\big(a\wedge \mu(A)\big),
 \end{align*}
 \item[(c)] $\rS(x+y,z)\le \rS(x,z)+yz$ for all $x,y,z$ iff for all $a,f,A,\mu$
  \begin{align*}
 \bIs\big(\mu,(f+a)\mI{A}\big)\le \bIs\big(\mu,f\mI{A}\big)+a\mu(A).
 \end{align*}
 \end{itemize}
 \end{corollary}

 \section{Some other properties of integral $\bIs$}
In this section, the solution to  Problems $2.31$ and $2.32$ of Borzowá-Molnárová,  Hal\v{c}inová and Hutník   $\cite{hutnik}$ is provided. 
We say that $f,g\colon X\to [0,1]$ are {\it comonotone} on $A\in\cF,$ if $\big(f(x)-f(y)\big)\big(g(x)-g(y)\big)\ge 0$ for all $x,y\in A.$ Clearly, if $f$ and $g$ are comonotone on $A,$ then for any real number $t$ either $\lbrace f\ge t\rbrace \subset \lbrace g\ge t\rbrace$ or $\lbrace g\ge t\rbrace \subset \lbrace f\ge t\rbrace.$

 \begin{pr}[Problem $2.31$]
Let $\mu\in\cM_{(X,\cA)}$ and $f,g\in\cF_{(X,\cA)}$ be comonotone functions. Characterize all the semicopulas $\rS$ for which
\begin{align}\label{r1}
\bIs(\mu,f\vee g)=\bIs(\mu,f)\vee \bIs(\mu,g).
\end{align}
\end{pr}

\begin{tw}\label{tw2}
The equality $\eqref{r1}$ holds for any semicopula $\rS\in\fS.$
\end{tw}
\begin{proof}
Since $f,g$ are comonotone,  $\lbrace f\ge t\rbrace \subset\lbrace g\ge t\rbrace$ or $\lbrace g\ge t\rbrace \subset\lbrace f\ge t\rbrace$ for all $t.$ Hence $\mu\big(\lbrace f\vee g\ge t\rbrace \big)=\mu\big(\lbrace f\ge t\rbrace \big)\vee \mu\big(\lbrace g\ge t\rbrace \big)$ and
\begin{align*}
\bIs\big(\mu,f\vee g\big)&=\sup_{t\in [0,1]} \rS\Big(t,\mu\big(\lbrace f\ge t\rbrace \big)\vee \mu\big(\lbrace g\ge t\rbrace \big)\Big)\\&=\sup_{t\in [0,1]} \left\{\rS\Big(t,\mu\big(\lbrace f\ge t\rbrace \big)\Big)\vee \rS\Big(t,\mu\big(\lbrace g\ge t\rbrace \big)\Big)\right\}\\&=\sup_{t\in [0,1]}  \rS\Big(t,\mu\big(\lbrace f\ge t\rbrace \big)\Big)\vee  \sup_{t\in [0,1]} \rS\Big(t,\mu\big(\lbrace g\ge t\rbrace \big)\Big)\\&=\bIs(\mu,f)\vee \bIs(\mu,g).
\end{align*}
The proof is complete.
\end{proof}

\begin{pr}[Problem $2.32$]
Let $\rS\in \fS$ be fixed. To describe all the commuting binary operators with the integral $\bIs$ (under the condition of comonotonicity of functions $f, g$), i.e., to  find all operators $\circ\colon [0,1]^2\to [0,1]$ such that
\begin{align}\label{r2}
\bIs(\mu,f\circ g)=\bIs(\mu,f)\circ \bIs(\mu,g)
\end{align}
for all $f,g\colon X\to [0,1]$ comonotone functions.
\end{pr}

We give an answer to this problem  for some class of operators $\circ$. 
\begin{tw}
Let $\rS\in \fS$ be fixed and let $\circ$ be either a~semicopula or $t$-coseminorm.
The equality $\eqref{r2}$ holds for all comonotone functions $f, g\in\cF_{(X,\cA)}$ iff  $\circ=\vee.$ 
\end{tw}

\begin{proof}
Put $f=\mI{A}=g$ and $A\in\cF.$ Thus $(f\circ g)(x)$ equals $1,$ if $x\in A$ and $0$ otherwise, which implies that  $\mu\big(\lbrace f\circ g\ge t\rbrace \big)=\mu(A)$ for $t\in (0,1]$ and $\mu\big(\lbrace f\circ g\ge 0\rbrace \big)=1.$  By $\eqref{r2}$
\begin{align*}
 \mu(A)=\mu(A)\circ \mu(A).
\end{align*}
In consequence, we have $x=x\circ x$ for all $x.$ If $\circ$ is a semicopula then $\circ=\wedge.$ Indeed, if $x\le y$, then 
\begin{align*}
 x=x\circ x\le x\circ y\le x\wedge y=x
 \end{align*} 
so $x\circ y=x\wedge y.$ Similar arguments apply to the case of $x>y$.  Obviously, if $\circ$ is a~$t$-coseminorm then $\circ=\vee.$

Now we show that if $\circ=\wedge$, then the equality (\ref{r2}) is not satisfied for all $\rS\in\fS.$ Suppose that $\rS=\Pi$ and $X=\lbrace a,b\rbrace$ with $\mu(\lbrace a\rbrace)=0.5,$ $\mu(\lbrace b\rbrace)=\beta\in [0,1]$ and $\mu(X)=1.$ 
 Furthermore, we assume that $f(a)=1,$ $f(b)=0.4,$ $g(a)=0.8$ and $g(b)=0.6.$ The functions $f,g$ are comonotone on $X.$ An~elementary algebra shows that
\begin{align*}
\bI_{\Pi}(\mu,f)=0.4\vee (1\cdot 0.5)=0.5,\\
\bI_{\Pi}(\mu,g)=0.6\vee  (0.8\cdot 0.5)=0.6,\\
\bI_{\Pi}\big(\mu,(f\wedge g)\big)=0.4\vee (0.8\cdot 0.5)=0.4,
\end{align*}
so $\bI_{\Pi}\big(\mu,(f\wedge g)\big)<\bI_{\Pi}(\mu,f)\wedge \bI_{\Pi}(\mu,g).$ The equality $\eqref{r2}$ holds  for 
$\circ=\vee$, which is a~consequence of  Theorem $\ref{tw2}.$
\end{proof}




\begin{thebibliography}{99}
\bibitem[1]{bas}B. Bassan, F. Spizzichino, Relations among univariate aging, bivariate aging and dependence for exchangeable lifetimes, J. Multivariate Analysis 93 (2005) 313–339.
\bibitem[2]{dur} F. Durante, C. Sempi, Semicopul\ae, Kybernetika 41 (2005) 315-328. 
\bibitem[3]{suarez} F. Su\'{a}rez García, P. Gil \'Alvarez, Two families of fuzzy integrals, Fuzzy Sets and Systems 18 (1986) 67-81.
\bibitem[4]{hutnik} J. Borzowá-Molnárová, L. Hal\v{c}inová, O. Hutník,  The smallest semicopula-based universal integrals I: properties and characterizations, Fuzzy Sets and Systems (2014) 
http://dx.doi.org/10.1016/j.fss.2014.09.0232014.
\bibitem[5]{klement2} E.P. Klement, R. Mesiar, E. Pap, Triangular norms, Kluwer Academic Publishers, Dordrecht, 2000.
\bibitem[6]{klement3} E.P. Klement, R. Mesiar, E. Pap, A~universal integral as common frame for Choquet and Sugeno integral, IEEE Transactions Fuzzy Sets and Systems 18 (2010) 178-187.
\bibitem[7]{lukasiewicz} J. Łukasiewicz, O~logice trójwartościowej (in Polish). Ruch filozoficzny  (1920) 170–171. English translation: On three-valued logic, in L. Borkowski (ed.), Selected works by Jan Łukasiewicz, North–Holland, Amsterdam, (1970) 87–88. 
\bibitem[8]{mesiar11} R. Mesiar, A. Stup\v{n}anov\'{a}, Open problems from the 12th International Conference on Fuzzy Set Theory and Its Applications, Fuzzy Sets and Systems (2014) http://dx.doi.org/10.1016/j.fss.2014.07.012
\bibitem[9]{ouyang3} Y. Ouyang, R. Mesiar, On the Chebyshev type inequality for seminormed fuzzy integral, Applied Mathematics Letters 22 (2009) 1810-1815.
\bibitem[10]{shilkret} N. Shilkret, Maxitive measure and integration, Indagationes Mathematicae 33 (1971) 109-116.
\bibitem[11]{sugeno1} M. Sugeno, Theory of Fuzzy Integrals and its Applications, Ph.D. Dissertation, Tokyo Institute of Technology, 1974.

\end{thebibliography}
\end{document}